\documentclass[11pt]{article}

\usepackage[utf8]{inputenc}
\usepackage[pdftex]{graphicx,xcolor}
\usepackage{array,bookmark,fullpage}
\usepackage{microtype}
\usepackage{amsmath,amssymb,amsthm,mathtools,bm,tikz-cd,stmaryrd,enumitem}
\usepackage[nottoc]{tocbibind}
\usepackage{hyperref}
\hypersetup{pdftex,colorlinks=true,allcolors=blue,unicode,psdextra}
\usepackage[capitalize]{cleveref}

\let\hash\#

\usepackage{mymacros}

\let\defn\textbf

\DeclareMathOperator\diam{diam}
\let\Dia\Diamond
\newcommand{\relmiddle}[1]{\nonscript\;\middle#1\nonscript\;}
\let\Up\relax
\DeclareMathOperator\Up{Up}
\DeclareMathOperator\Filt{Filt}
\DeclareMathOperator\Idl{Idl}
\DeclareMathOperator\Coidl{Coidl}
\DeclareMathOperator\PFilt{PFilt}
\mathlig{|>}{\rhd}

\numberwithin{equation}{section}

\begin{document}

\title{Notes on quasi-Polish spaces}
\author{Ruiyuan Chen}
\date{}
\maketitle

\begin{abstract}
Quasi-Polish spaces were introduced by de~Brecht as a possibly non-Hausdorff generalization of Polish spaces sharing many of their descriptive set-theoretic properties.  We give a self-contained exposition of the basic theory of quasi-Polish spaces, based on their ``logical'' characterization as $\*\Pi^0_2$ subspaces of countable powers of Sierpinski space, with several new proofs emphasizing this point of view as well as making more extensive use of Baire category techniques.
\renewcommand\thefootnote{}
\end{abstract}

\section{Introduction}
\label{sec:intro}

Polish spaces, i.e., separable, completely metrizable topological spaces, are the central setting for classical descriptive set theory.
\defn{Quasi-Polish spaces} are, informally, a certain well-behaved generalization of Polish spaces not required to obey any separation axioms beyond $T_0$.
Quasi-Polish spaces were introduced by de~Brecht \cite{deB}, who showed that they satisfy analogs of many of the basic descriptive set-theoretic properties of Polish spaces.
Quasi-Polish spaces also admit some natural constructions with no good analogs for Polish spaces (e.g., the lower powerspace of closed sets; see \cref{sec:lowpow}); thus, quasi-Polish spaces can be useful to consider even when one is initially interested only in the Polish context.

In \cite{deB}, quasi-Polish spaces are defined as second-countable, completely quasi-metrizable spaces, where a \defn{quasi-metric} is a generalization of a metric that is not required to obey the symmetry axiom $d(x, y) = d(y, x)$.
This is a natural generalization of the usual definition of Polish spaces as second-countable, completely metrizable spaces.
It is then proved that

\begin{theorem}[{\cite[Theorem~24]{deB}}]
\label{thm:qpol-qmet}
Quasi-Polish spaces are precisely the homeomorphic copies of $\*\Pi^0_2$ subsets of $\#S^\#N$.
\end{theorem}

Here $\#S = \{0, 1\}$ is the \defn{Sierpinski space}, with $\{1\}$ open but not closed, and can be thought of as the topological space with a ``generic'' open set (namely $\{1\}$).  Similarly, the product $\#S^\#N$ can be thought of as the space with countably many ``generic'' open sets (the subbasic ones).  In non-metrizable spaces such as $\#S^\#N$, $G_\delta$ sets are not so well-behaved since they may not include all closed sets; thus it is convenient to alter the classical definition of $\*\Pi^0_2$ to mean all sets of the form
\begin{align*}
\bigcap_n (\neg U_n \cup V_n) = \{x \mid \forall n\, (x \in U_n \implies x \in V_n)\}
\end{align*}
for countably many open sets $U_n, V_n$.  Note that the above set can be read as ``the set of all $x$ where the implications $U_n \Rightarrow V_n$ hold''.  Thus, \cref{thm:qpol-qmet} can be read as
\begin{align*}
\text{quasi-Polish space} ={}& \text{``space with countably many generic open sets, and} \\ &\text{ countably many relations imposed between them''}.
\end{align*}

The purpose of these notes is to give a concise, self-contained account of the basic theory of quasi-Polish spaces from this point of view.  That is, we take \cref{thm:qpol-qmet} as a definition; in fact, we will not mention quasi-metrics at all.  Whenever we show that a space is quasi-Polish, we will give an explicit $\*\Pi^0_2$ definition of it as a subspace of a known quasi-Polish space (such as $\#S^\#N$).  Our exposition also makes no reference to domain theory or various other classes of spaces inspired by computability theory (see e.g., \cite[\S9]{deB}).  It is hoped that such an approach will be easily accessible to descriptive set theorists and others familiar with the classical theory of Polish spaces.

We would like to stress that these notes contain essentially no new results.
Most of the results we discuss are from the papers \cite{deB} and \cite{dBK}, or are easy generalizations of classical results for Polish spaces.
Whenever possible, we give a reference to the same (or equivalent) result in one of these papers.
However, the proofs we give are usually quite different from those referenced, reflecting our differing point of view.

As our main goal is to give a concise exposition of the basic results about quasi-Polish spaces, we have neglected to treat many other relevant topics, e.g., local compactness \cite[\S8]{deB}, the Hausdorff--Kuratowski theorem and difference hierarchy \cite[\S13]{deB}, Hurewicz's theorem for non-quasi-Polish $\*\Pi^1_1$ sets \cite{dB2}, and upper powerspaces \cite{dBK}, among others.  For the same reason, we do not include a comprehensive bibliography, for which we refer the reader to the aforementioned papers.

Finally, we remark that our approach is heavily inspired by the correspondence between quasi-Polish spaces and countably (co)presented locales \cite{Hec}.  A \defn{locale} is, informally, a topological space without an underlying set, consisting only of an abstract lattice of ``open sets''.  The definition of quasi-Polish spaces in terms of countably many ``generators and relations'' for their open set lattices leads naturally to the idea of forgetting about the points altogether and regarding the open sets as an abstract lattice, i.e., replacing spaces with locales.  In what follows, we will not refer explicitly to the localic viewpoint; however, the reader who is familiar with locale theory will no doubt recognize its influence in several places (most notably \cref{sec:posite}).

\medskip
\textit{Acknowledgments.}  I would like to thank Alexander Kechris and Matthew de~Brecht for providing some comments on earlier revisions of these notes.

\section{Basic definitions}
\label{sec:defns}

Recall that on an arbitrary topological space $X$, the \defn{specialization preorder} is given by
\begin{align*}
x \le y \iff x \in \-{\{y\}} \iff \forall \text{ (basic) open $U$}\, (x \in U \implies y \in U).
\end{align*}
The specialization preorder is a partial order iff $X$ is $T_0$, and is discrete iff $X$ is $T_1$.  Open sets are upward-closed; closed sets are downward-closed.  The principal ideal
\begin{align*}
\down x := \{y \in X \mid y \le x\}
\end{align*}
generated by a point $x \in X$ coincides with its closure $\-{\{x\}}$.

The \defn{Sierpinski space} $\#S = \{0, 1\}$ has $\{1\}$ open but not closed; the specialization order is thus given by $0 < 1$.

We will be concerned with product spaces $\#S^I$ and their subspaces, especially for $I$ countable.  Whenever convenient, we identify $\#S^I$ with $\@P(I)$, the powerset of $I$; note that the specialization order on $\#S^I$ corresponds to inclusion of subsets.  A basis of open sets in $\#S^I$ consists of the sets
\begin{align*}
\up s := \{x \in \@P(I) \mid s \subseteq x\} \qquad\text{for finite $s \subseteq I$}.
\end{align*}

Given an arbitrary topological space $X$, not necessarily metrizable, we define the \defn{Borel hierarchy} on $X$ as follows; this definition is due to Selivanov \cite{Sel}.  The $\*\Sigma^0_1$ sets are the open sets.  For an ordinal $\xi > 1$, the $\*\Sigma^0_\xi$ sets are those of the form
\begin{align*}
\bigcup_{n \in \#N} (A_n \setminus B_n) \qquad\text{for $A_n, B_n \in \*\Sigma^0_{\zeta_n}(X)$, $\zeta_n < \xi$}
\end{align*}
(we write $\*\Sigma^0_\xi(X)$ for the set of $\*\Sigma^0_\xi$ sets in $X$).  It is easy to see by induction that for $\xi > 2$, we may take $A_n = X$ above, as in the usual definition of the Borel hierarchy (in the metrizable case).  The $\*\Pi^0_\xi$ sets are the complements of the $\*\Sigma^0_\xi$ sets, and the $\*\Delta^0_\xi$ sets are those which are both $\*\Sigma^0_\xi$ and $\*\Pi^0_\xi$; these are denoted $\*\Pi^0_\xi(X), \*\Delta^0_\xi(X)$ respectively.  A set is \defn{Borel} if it is $\*\Sigma^0_\xi$ for some $\xi < \omega_1$.  We have the usual picture of the Borel hierarchy:
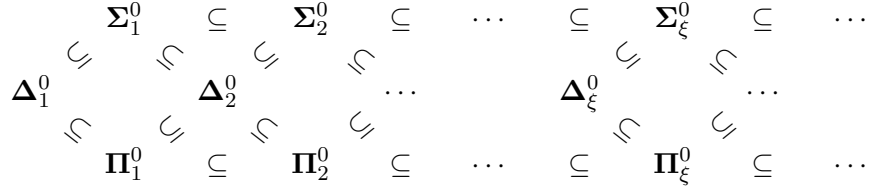
\begin{equation*}
\begin{tikzcd}[row sep=0.5em,column sep=1em,every arrow/.style={phantom,"\subseteq" sloped}]
& \*\Sigma^0_1 \drar & \subseteq & \*\Sigma^0_2 \drar & \subseteq & \dotsb & \subseteq & \*\Sigma^0_\xi \drar & \subseteq & \dotsb \\
\*\Delta^0_1 \urar \drar && \*\Delta^0_2 \urar \drar && \dotsb && \*\Delta^0_\xi \urar \drar && \dotsb \\
& \*\Pi^0_1 \urar & \subseteq & \*\Pi^0_2 \urar & \subseteq & \dotsb & \subseteq & \*\Pi^0_\xi \urar & \subseteq & \dotsb
\end{tikzcd}
\end{equation*}

Of particular note are the $\*\Pi^0_2$ sets
\begin{align*}
\bigcap_n (\neg U_n \cup V_n) = \{x \in X \mid \forall n\, (x \in U_n \implies x \in V_n)\}
\end{align*}
for $U_n, V_n$ open; they are the result of ``imposing countably many relations between open sets''.  The following are immediate:

\begin{proposition}[{\cite[Proposition~8]{deB}}]
Points in a first-countable $T_0$ space are $\*\Pi^0_2$.  \qed
\end{proposition}

\begin{proposition}[{\cite[Proposition~9]{deB}}]
The specialization preorder on a second-countable space is $\*\Pi^0_2$.  Hence, the equality relation on a second-countable $T_0$ space is $\*\Pi^0_2$.  \qed
\end{proposition}

A \defn{quasi-Polish space} $X$ is a homeomorphic copy of a $\*\Pi^0_2$ subspace of $\#S^I$ for some countable $I$, equivalently of $\#S^\#N$ \cite[Theorem~24]{deB}.  In other words, it is the result of imposing countably many relations between countably many ``generic'' open sets (the subbasic open sets $\up\{i\} \subseteq \#S^\#N$).  This is made more explicit by the following definitions.

For a topological space $X$ and a collection $\@U$ of open sets in $X$, define
\begin{align*}
e_\@U : X &--> \#S^\@U \\
x &|--> \{U \in \@U \mid x \in U\}.
\end{align*}
$e_\@U$ is continuous, and is an embedding if $X$ is $T_0$ and $\@U$ is a subbasis, in which case we call $e_\@U$ the \defn{canonical embedding} (with respect to $\@U$).
A \defn{countable copresentation} of a $T_0$ space $X$ consists of a countable subbasis $\@U$ for $X$ together with a $\*\Pi^0_2$ definition of $e_\@U(X) \subseteq \#S^\@U$.  Thus, $X$ is quasi-Polish iff it is \defn{countably copresented} (has a countable copresentation).

Many properties of quasi-Polish spaces can also be established with no extra effort for the more general class of \defn{countably correlated spaces}, which are homeomorphic copies of $\*\Pi^0_2$ subspaces of $\#S^I$ for arbitrary index sets $I$.

Recall that a \defn{Polish space} is a separable, completely metrizable topological space, while a \defn{standard Borel space} is a set equipped with the Borel $\sigma$-algebra of some Polish topology.  See \cite{Kec} for basic descriptive set theory on Polish spaces.  We will show below (\cref{thm:pol-qpol}) that quasi-Polish spaces are a generalization of Polish spaces; hence, most of the results that follow are generalizations of their classical analogs for Polish spaces.

\section{Basic properties}
\label{sec:limcolim}

\begin{proposition}
\label{thm:qpol-sborel}
Quasi-Polish spaces are standard Borel, and can be made Polish by adjoining countably many closed sets to the topology.
\end{proposition}
\begin{proof}
If $X \subseteq \#S^\#N$ is $\*\Pi^0_2$, then $X \subseteq 2^\#N$ is $G_\delta$, and is the result of adjoining the complements of the (sub)basic open sets in $\#S^\#N$ (whence the Borel $\sigma$-algebras agree).
\end{proof}

\begin{proposition}[{\cite[Theorem~22]{deB}}]
\label{thm:qpol-pi02}
A $\*\Pi^0_2$ subspace of a quasi-Polish space is quasi-Polish.  Similarly for countably correlated spaces.
\end{proposition}
\begin{proof}
Obvious.
\end{proof}

\begin{proposition}[{\cite[Corollary~43]{deB}}]
\label{thm:qpol-prod}
A countable product of quasi-Polish spaces is quasi-Polish.  Similarly for countably correlated spaces.
\end{proposition}
\begin{proof}
If $X_i \subseteq \#S^{I_i}$ are $\*\Pi^0_2$, then so is $\prod_i X_i = \bigcap_i p_i^{-1}(X_i) \subseteq \prod_i \#S^{I_i}$ where $p_i : \prod_i \#S^{I_i} -> \#S^{I_i}$ is the $i$th projection.
\end{proof}

For any topological space $X$, let
\begin{align*}
X_\bot := X \sqcup \{\bot\}
\end{align*}
where the open sets are those in $X$ together with all of $X_\bot$.  (Thus, $\bot$ is a newly adjoined least element in the specialization preorder, often thought of as ``undefined''.)

\begin{proposition}
\label{thm:qpol-bot}
If $X$ is quasi-Polish, then so is $X_\bot$.  Similarly for countably correlated spaces.
\end{proposition}
\begin{proof}
Suppose $X \subseteq \#S^I$ is $\*\Pi^0_2$.  Then
\begin{align*}
X_\bot &\cong \{(x, a) \in \#S^I \times \#S \mid (x, a) = (0, 0) \OR (x \in X \AND a = 1)\} \\
x &|-> (x, 1) \\
\bot &|-> (0, 0).
\qedhere
\end{align*}
\end{proof}

\begin{proposition}[{\cite[Corollary~43]{deB}}]
\label{thm:qpol-disj}
A countable disjoint union of quasi-Polish spaces is quasi-Polish.  Similarly for countably correlated spaces.
\end{proposition}
\begin{proof}
Let $X_i$ be quasi-Polish (or countably correlated).  Then
\begin{align*}
\bigsqcup_i X_i &\cong \{(x_i)_i \in \prod_i (X_i)_\bot \mid \exists i\, (x_i \in X_i) \AND \forall i \ne j\, (x_i \in X_i \implies x_j = \bot)\} \\
X_i \ni x &|-> (x \text{ if $j = i$, else } \bot)_j.
\qedhere
\end{align*}
\end{proof}

A topological space $X$ is \defn{$\sigma$-locally quasi-Polish} if it has a countable cover $\@U$ by open quasi-Polish subspaces.

\begin{proposition}
\label{thm:lqpol-qpol}
$\sigma$-Locally quasi-Polish spaces are quasi-Polish.
\end{proposition}
\begin{proof}
Let $X, \@U$ be as above.  Then
\begin{align*}
X &\cong \left\{(x_U)_U \in \prod_{U \in \@U} U_\bot \relmiddle| \begin{aligned}
&\exists U\, (x_U \in U) \AND \\
&\forall U, V\, \left(\begin{aligned}
    x_U \in U \cap V &\implies x_V = x_U \AND \\
    x_U \in U \setminus V &\implies x_V = \bot
\end{aligned}\right)
\end{aligned}\right\} \\
x &|-> (x \text{ if $x \in U$, else } \bot)_U.
\qedhere
\end{align*}
\end{proof}

\section{Subspaces}
\label{sec:sub}

Recall \cite[3.11]{Kec} that a subspace of a Polish space is Polish iff it is $G_\delta$.  An analogous fact holds for quasi-Polish spaces.

\begin{theorem}[{\cite[Theorem~21]{deB}}]
\label{thm:qpol-sub-pi02}
Let $X$ be a second-countable $T_0$ space and $Y \subseteq X$ be a countably correlated subspace.  Then $Y \subseteq X$ is $\*\Pi^0_2$.
\end{theorem}

The proof we give consists essentially of applying the following simple fact in universal algebra to the lattice of open sets of $Y$.  Given any countably presented algebraic structure $A$ (e.g., group, ring, \ldots) and countably many generators $a_0, a_1, \dotsc \in A$, there is a countable presentation of $A$ using only those generators.  To see this: let $A = \ang{b_0, b_1, \dotsc \mid R}$ be any countable presentation; write each $b_j$ as some word $w_j$ in the $a_i$, and substitute $b_j |-> w_j$ into $R$ to get a set of relations $S$ in the $a_i$; write each $a_i$ as some word $v_i$ in the $b_j$, and substitute $b_j |-> w_j$ into $v_i$ to get a word $v_i'$ in the $a_k$ (which evaluates to $a_i$ in $A$); then $A = \ang{a_0, a_1, \dotsc \mid S \cup \{a_i = v_i'\}_i}$.

\begin{proof}
Let $f : Y -> \#S^I$ be an embedding with $\*\Pi^0_2$ image, say $f(Y) = \bigcap_n (\neg U_n \cup V_n)$ where $U_n, V_n \subseteq \#S^I$ are open.  Thus each $U_n, V_n$ is a union of basic open sets:
\begin{align*}
U_n = \bigcup_{\up s \subseteq U_n} \up s, &&
V_n = \bigcup_{\up s \subseteq V_n} \up s,
\end{align*}
where $s$ runs over finite subsets of $I$.  For each $i \in I$, let $W_i \subseteq X$ be open such that
\begin{align*}
W_i \cap Y = f^{-1}(\up \{i\}).
\end{align*}
For each finite $s \subseteq I$, put $W_s := \bigcap_{i \in s} W_i$, so that
\begin{align*}
W_s \cap Y = f^{-1}(\up s).
\end{align*}
Let $\@W$ be a countable subbasis of open sets in $X$.  We claim that
\begin{align*}
Y = \underbrace{\bigcap_n (\neg \bigcup_{\up s \subseteq U_n} W_s \cup \bigcup_{\up s \subseteq V_n} W_s)}_A \cap \bigcap_{W \in \@W} \underbrace{\neg (W \triangle \bigcup_{f^{-1}(\up s) \subseteq W} W_s)}_{B_W}.
\end{align*}
$\subseteq$ is straightforward.  To prove $\supseteq$, let $x \in X$ belong to the right-hand side.  Put
\begin{align*}
z := \{i \in I \mid x \in W_i\} \in \#S^I.
\end{align*}
Using $x \in A$, we easily have $z \in \bigcap_n (\neg U_n \cup V_n) = f(Y)$.  Let $z = f(y)$.  For each $W \in \@W$, we have
\begin{align*}
\begin{aligned}
y \in W = \bigcup_{f^{-1}(\up s) \subseteq W} f^{-1}(\up s)
&\iff z \in \bigcup_{f^{-1}(\up s) \subseteq W} \up s &&\text{since $f$ is an embedding} \\
&\iff x \in \bigcup_{f^{-1}(\up s) \subseteq W} W_s &&\text{by definition of $z$} \\
&\iff x \in W &&\text{since $x \in B_W$}.
\end{aligned}
\end{align*}
Thus since $X$ is $T_0$, $x = y \in Y$.
\end{proof}

\begin{corollary}[{\cite[Theorem~23]{deB}}]
\label{thm:qpol-sub=pi02}
Let $X$ be a quasi-Polish space.  A subspace $Y \subseteq X$ is quasi-Polish iff it is $\*\Pi^0_2$.  \qed
\end{corollary}

\begin{corollary}
\label{thm:qpol=2cnt+ccor}
A space $X$ is quasi-Polish iff it is second-countable and countably correlated.
\end{corollary}
\begin{proof}
If $X$ is second-countable and countably correlated, then letting $\@U$ be a countable subbasis, the canonical embedding $e_\@U : X -> \#S^\@U$ (see \cref{sec:defns}) has $\*\Pi^0_2$ image by \cref{thm:qpol-sub-pi02}.
\end{proof}

\section{Polish spaces}
\label{sec:polish}

\begin{theorem}[{\cite{deB}\footnote{%
Since \cite{deB} defines quasi-Polish spaces in terms of complete quasi-metrics, which generalize complete metrics, \cref{thm:pol-qpol} is trivial according to the definitions in \cite{deB}.  In fact, the content of \cref{thm:pol-qpol} is contained in the proofs of \cite[Theorems~19--21]{deB} (which establish that their definition of quasi-Polish space implies ours).%
}}]
\label{thm:pol-qpol}
Polish spaces are quasi-Polish.
\end{theorem}

\begin{proof}[Proof 1]
First, we note

\begin{lemma}
\label{thm:R-qpol}
$\#R$ is quasi-Polish.
\end{lemma}
\begin{proof}
We have
\begin{align*}
\#R &\cong \{(A, B) \in \#S^\#Q \times \#S^\#Q \mid (A, B) \in \@P(\#Q) \times \@P(\#Q) \text{ is a Dedekind cut}\} \\
&= \left\{(A, B) \in \#S^\#Q \times \#S^\#Q \relmiddle| \begin{aligned}
&A \ne \emptyset \AND B \ne \emptyset \AND \\
&\forall p < q \in A\, (p \in A) \AND \forall p > q \in B\, (p \in B) \AND \\
&\forall p \in A\, \exists q > p\, (q \in A) \AND \forall p \in B\, \exists q < p\, (q \in B) \AND \\
&A \cap B = \emptyset \AND \\
&\forall p < q \in \#Q\, (p \in A \OR q \in B)
\end{aligned}\right\} \\
r &|-> (\{q \in \#Q \mid q < r\}, \{q \in \#Q \mid q > r\}).
\qedhere
\end{align*}
\end{proof}

Now let $X$ be a Polish space with compatible complete metric $d$ and $D \subseteq X$ be a countable dense subset.  Then using a standard construction of the completion of $D$,
\begin{align*}
X &\cong \{f \in \#R^D \mid f \text{ is a Katětov function} \AND \inf f = 0\} \\
&= \left\{f \in \#R^D \relmiddle| \begin{aligned}
&\forall x, y \in D\, (f(x) - f(y) \le d(x, y) \le f(x) + f(y)) \AND \\
&\forall n \ge 1\, \exists x \in D\, (f(x) < 1/n)
\end{aligned}\right\} \\
x &|-> d(x, -).
\qedhere
\end{align*}
\end{proof}

Sometimes it is useful to have a countable copresentation of a Polish space derived from a countable basis instead of a countable dense subset (as in the above proof).  This is provided by the following alternative proof, which is also more direct in that it avoids first showing that $\#R$ is quasi-Polish.

\begin{proof}[Proof 2]
Let $X$ be a Polish space with compatible complete metric $d$.  Let $\@U$ be a countable basis of open sets in $X$, closed under binary intersections (so containing $\emptyset$).  For $U \in \@U$ and $r > 0$, put
\begin{align*}
[U]_r := \{x \in X \mid \exists y \in U\, (d(x, y) < r)\},
\end{align*}
the $r$-neighborhood of $U$.  We claim that the canonical embedding $e_\@U : X -> \#S^\@U$ has image
\begin{align*}
e_\@U(X) = \left\{\@A \subseteq \@U \relmiddle| \begin{aligned}
&\emptyset \not\in \@A \AND \\
&\forall U, V \in \@U\, (U \cap V \in \@A \iff U, V \in \@A) \AND \\
&\forall n \ge 1\, \exists U \in \@A\, (\diam(U) < 1/n) \AND \\
&\forall U \in \@A\, \exists n \ge 1,\, V \in \@A\, ([V]_{1/n} \subseteq U)
\end{aligned}\right\}.  \tag{$*$}
\end{align*}
$\subseteq$ is straightforward.  To prove $\supseteq$, let $\@A$ belong to the right-hand side; we must find $x \in X$ such that $x \in U \iff U \in \@A$ for all $U \in \@U$.  By the first three conditions on the right-hand side ($*$), $\@A$ is a Cauchy filter base.  Let $x$ be its limit, i.e.,
\begin{align*}
\{x\} = \bigcap_{U \in \@A} \-U.
\end{align*}
For $U \in \@U$ such that $x \in U$, since $\@A$ is Cauchy, there is some $V \in \@A$ such that $V \subseteq U$, whence $U \in \@A$ by the second condition on the right-hand side ($*$).
Conversely, for $U \in \@A$, by the fourth condition on the right-hand side ($*$) there is some $n \ge 1$ and $V \in \@A$ with $[V]_{1/n} \subseteq U$, whence $x \in \-V \subseteq U$.
\end{proof}

\begin{corollary}
\label{thm:qpol-reg-pol}
A topological space $X$ is Polish iff it is quasi-Polish and regular.
\end{corollary}
\begin{proof}
If $X$ is quasi-Polish and regular, then $X$ is second-countable and $T_3$, whence by the Urysohn metrization theorem, $X$ is metrizable; letting $\^X$ be a completion of $X$ with respect to a compatible metric, $\^X$ is Polish, and $X \subseteq \^X$ is $\*\Pi^0_2$ by \cref{thm:qpol-sub=pi02}, hence Polish.
\end{proof}

\section{Change of topology}

\begin{theorem}[{\cite[Theorem~73]{deB}}]
\label{thm:qpol-dissolv-delta02}
Let $X$ be a quasi-Polish space and $A_0, A_1, \dotsc \subseteq X$ be countably many $\*\Delta^0_2$ sets.  Then the space $X'$ given by $X$ with $A_0, A_1, \dotsc$ adjoined to its topology is quasi-Polish.  Similarly for countably correlated spaces.
\end{theorem}
\begin{proof}
We have
\begin{align*}
X' &\cong \{(x, (a_n)_n) \in X \times \#S^\#N \mid \forall n\, (x \in A_n \iff a_n = 1)\} \\
x &|-> (x, (1 \text{ if $x \in A_n$, else } 0)_n).
\qedhere
\end{align*}
\end{proof}

\begin{remark}
As noted in \cite[paragraph before Lemma~72]{deB}, given a Polish space, adjoining $\*\Delta^0_2$ sets which are not closed might result in a non-metrizable space.
\end{remark}

We also have a converse to \cref{thm:qpol-dissolv-delta02} in the case of a single set:

\begin{proposition}
Let $X$ be a quasi-Polish space and $A \subseteq X$ be such that the space $X'$ given by $X$ with $A$ adjoined to its topology is quasi-Polish.  Then $A \subseteq X$ is $\*\Delta^0_2$.
\end{proposition}
\begin{proof}
Consider the embedding $e : X' -> X \times \#S$ from the proof of \cref{thm:qpol-dissolv-delta02}.  Since $X'$ is quasi-Polish, $e(X') \subseteq X \times \#S$ is $\*\Pi^0_2$ by \cref{thm:qpol-sub=pi02}.  Thus $A = \{x \in X \mid (x, 1) \in e(X')\}$ is $\*\Pi^0_2$, as is $\neg A = \{x \in X \mid (x, 0) \in e(X')\}$.
\end{proof}

\begin{lemma}[{\cite[Lemma~72]{deB}}]
\label{thm:qpol-dissolv-invlim}
Let $X$ be a quasi-Polish space and $\tau_0, \tau_1, \dotsc$ be finer quasi-Polish topologies on $X$.  Then the topology $\tau$ generated by $\tau_0, \tau_1, \dotsc$ is quasi-Polish.
\end{lemma}
\begin{proof}
We have
\begin{align*}
(X, \tau) &\cong \{(x, (x_i)_i) \in X \times \prod_i (X, \tau_i) \mid \forall i\, (x = x_i)\} \\
x &|-> (x, (x)_i).
\qedhere
\end{align*}
\end{proof}

\begin{theorem}[{\cite[Theorem~74]{deB}}]
Let $X$ be a quasi-Polish space and $A_0, A_1, \dotsc \in \*\Sigma^0_\xi(X)$.  Then there is a finer quasi-Polish topology on $X$ containing each $A_i$ and contained in $\*\Sigma^0_\xi(X)$.
\end{theorem}
\begin{proof}
By \cref{thm:qpol-dissolv-invlim} it suffices to consider the case of a single $A \in \*\Sigma^0_\xi(X)$.  We induct on $\xi$.  The case $\xi = 1$ is trivial, so assume $\xi > 1$.  Write $A = \bigcup_i (B_i \setminus C_i)$ where $B_i, C_i \in \*\Sigma^0_{\zeta_i}(X)$ for $\zeta_i < \xi$.  By the induction hypothesis, there are finer quasi-Polish topologies $\tau_i \subseteq \*\Sigma^0_{\zeta_i}(X)$ such that $B_i, C_i \in \tau_i$.  Then each $B_i \setminus C_i \in \*\Delta^0_2(X, \tau_i)$, so by \cref{thm:qpol-dissolv-delta02}, the topology $\tau_i'$ generated by $\tau_i$ and $B_i \setminus C_i$ is quasi-Polish.  Now by \cref{thm:qpol-dissolv-invlim}, the topology $\tau$ generated by the $\tau_i'$ is quasi-Polish.  Clearly $\tau_i' \subseteq \*\Delta^0_{\zeta_i+1}(X) \subseteq \*\Sigma^0_\xi(X)$, whence $\tau \subseteq \*\Sigma^0_\xi(X)$; and $A \in \tau$.
\end{proof}

\section{Baire category}
\label{sec:baire}

Recall \cite[\S8]{Kec} that a topological space $X$ is \defn{Baire} if the intersection of countably many dense open sets in $X$ is dense; and that a subset $A \subseteq X$ is \defn{comeager} if it contains a countable intersection of dense open sets, \defn{meager} if its complement is comeager, and \defn{Baire-measurable} (or has the \defn{Baire property}) if it differs from an open set by a meager set.

In the non-metrizable setting, it is useful to note the following:

\begin{proposition}
\label{thm:dense-pi02-gdelta}
Let $X$ be a topological space, $G \subseteq X$ be a dense $\*\Pi^0_2$ subset.  Then $G$ is comeager.  Thus, $A \subseteq X$ is comeager iff it contains a countable intersection of dense $\*\Pi^0_2$ sets.
\end{proposition}
\begin{proof}
Let $G = \bigcap_n (\neg U_n \cup V_n)$ where $U_n, V_n \subseteq X$ are open.  Since $G$ is dense, so is each $\neg U_n \cup V_n$, i.e., $X = \-{\neg U_n \cup V_n} = \-{\neg U_n} \cup \-{V_n} = \neg U_n \cup \-{V_n}$; since $\-{V_n}$ is closed, this implies $X = (\neg U_n)^\circ \cup \-{V_n} \subseteq \-{(\neg U_n)^\circ \cup V_n}$.  So the $(\neg U_n)^\circ \cup V_n \subseteq X$ are dense open sets whose intersection is contained in $G$.
\end{proof}

A space $X$ is \defn{completely Baire} if every closed subspace $Y \subseteq X$ is Baire.

\begin{proposition}[see {\cite[4.1]{dB2}}]
\label{thm:herebaire}
Let $X$ be a topological space.  The following are equivalent:
\begin{enumerate}
\item[(i)]  Every $\*\Pi^0_2$ subspace $Y \subseteq X$ is Baire.
\item[(ii)]  $X$ is completely Baire.
\item[(iii)]  Every nonempty closed $F \subseteq X$ is non-meager in $F$.
\end{enumerate}
\end{proposition}
\begin{proof}
Clearly (i)$\implies$(ii)$\implies$(iii).  Assume (iii), and let $Y \subseteq X$ be $\*\Pi^0_2$; we show that $Y$ is Baire.  Let $W_n \subseteq X$ be open sets dense in $Y$; we must show that $\bigcap_n W_n$ is dense in $Y$.  Let $U \subseteq X$ be open with $U \cap Y \ne \emptyset$; we must show that $U \cap Y \cap \bigcap_n W_n \ne \emptyset$.  Put $F := \-{U \cap Y}$; clearly $F \ne \emptyset$.  Since $U \cap Y$ is $\*\Pi^0_2$ and dense in $F$, by \cref{thm:dense-pi02-gdelta} there are $V_n \subseteq F$ dense open in $F$ with $\bigcap_n V_n \subseteq U \cap Y$.  Each $W_n$ is dense in $U \cap Y$, hence also in $F$, so by (iii), $\emptyset \ne \bigcap_n V_n \cap \bigcap_n W_n \subseteq U \cap Y \cap \bigcap_n W_n$, as desired.
\end{proof}

\begin{theorem}[Baire category theorem {\cite[Corollary~52]{deB}}]
Countably correlated spaces are (completely) Baire.
\end{theorem}
\begin{proof}
By \cref{thm:herebaire}, it is enough to show that every nonempty closed $F \subseteq \#S^I$ is non-meager in $F$.  Let $U_n \subseteq \#S^I$ be open and dense in $F$; we must show that $F \cap \bigcap_n U_n \ne \emptyset$.
We will find finite $s_0 \subseteq s_1 \subseteq \dotsb \subseteq I$ and $x_n \in F \cap \up s_n$.
Let $s_0 := \emptyset$; then $F \cap \up s_0 = F \ne \emptyset$, so there is some $x_0 \in F \cap \up s_0$.  Given $s_n, x_n$ such that $x_n \in F \cap \up s_n \ne \emptyset$, since $U_n$ is dense in $F$, we have $F \cap \up s_n \cap U_n \ne \emptyset$, so there is some $x_{n+1} \in F \cap \up s_n \cap U_n$, whence there is some basic open $\up s_{n+1} \subseteq \up s_n \cap U_n$ such that $x_{n+1} \in \up s_{n+1}$, whence $s_n \subseteq s_{n+1}$ and $x_{n+1} \in F \cap \up s_{n+1}$.
Put $x := \bigcup_n s_n$.  Then $x \in \up s_{n+1} \subseteq U_n$ for each $n$, and $x = \lim_{n -> \infty} x_n \in F$, whence $F \cap \bigcap_n U_n \ne \emptyset$, as desired.
\end{proof}

As for Polish spaces \cite[\S8.J]{Kec}, we also have a well-behaved theory of ``fiberwise'' Baire category, i.e., category quantifiers, for quasi-Polish spaces.  We will state this in a more general context.

Let $f : X -> Y$ be a function between sets $X, Y$, such that for each $y \in Y$, the fiber $f^{-1}(y) \subseteq X$ is equipped with a topology.  For a subset $A \subseteq X$, put
\begin{align*}
\exists^*_f(A) &:= \{y \in Y \mid A \cap f^{-1}(y) \text{ is not meager in } f^{-1}(y)\} \subseteq Y, \\
\forall^*_f(A) &:= \{y \in Y \mid A \cap f^{-1}(y) \text{ is comeager in } f^{-1}(y)\} = \neg \exists^*_f(\neg A) \subseteq Y.
\end{align*}
A subset $U \subseteq X$ is \defn{$f$-fiberwise open} if $U \cap f^{-1}(y)$ is open in $f^{-1}(y)$ for each $y \in Y$; notions such as \defn{$f$-fiberwise Baire}, \defn{$f$-fiberwise Baire-measurable} are defined similarly.  A family $\@W$ of $f$-fiberwise open subsets of $X$ is a \defn{$f$-fiberwise weak basis} for a $f$-fiberwise open $U \subseteq X$ if for every $y \in Y$ and nonempty open $V \subseteq U \cap f^{-1}(y)$, there is some $\@W \ni W \subseteq U$ with $\emptyset \ne W \cap f^{-1}(y) \subseteq V$.

\begin{proposition}[{see \cite[8.27]{Kec}}]
\label{thm:bairequant-ops}
Let $f : X -> Y$ be as above.
\begin{enumerate}
\item[(i)]  If $X$ is $f$-fiberwise Baire, then for $f$-fiberwise open $U \subseteq X$,
\begin{align*}
\exists^*_f(U) = f(U).
\end{align*}
\item[(ii)]  For countably many $A_n \subseteq X$,
\begin{align*}
\exists^*_f(\bigcup_n A_n) = \bigcup_n \exists^*_f(A_n).
\end{align*}
\item[(iii)]  If $X$ is $f$-fiberwise Baire, then for $f$-fiberwise open $U \subseteq X$, $f$-fiberwise Baire-measurable $A \subseteq X$, and a $f$-fiberwise weak basis $\@W$ for $U$,
\begin{align*}
\exists^*_f(U \setminus A) = \bigcup_{\@W \ni W \subseteq U} (f(W) \setminus \exists^*_f(W \cap A)).
\end{align*}
\end{enumerate}
\end{proposition}
\begin{proof}
(i) and (ii) are straightforward.  For (iii), if $y \in \exists^*_f(U \setminus A)$, i.e., $(U \setminus A) \cap f^{-1}(y)$ is non-meager in $f^{-1}(y)$, then letting (by the Baire property) $(U \setminus A) \cap f^{-1}(y) = V \triangle M$ where $V \subseteq U \cap f^{-1}(y)$ is open and $M \subseteq f^{-1}(y)$ is meager, we have some $\@W \ni W \subseteq U$ with $\emptyset \ne W \cap f^{-1}(y) \subseteq V$, whence $y \in f(W)$, and $W \cap A \cap f^{-1}(y) \subseteq V \cap A \cap f^{-1}(y) \subseteq M$, whence $y \not\in \exists^*_f(W \cap A)$.  Conversely, if $\@W \ni W \subseteq U$ with $y \in f(W) \setminus \exists^*_f(W \cap A)$, i.e., $W \cap f^{-1}(y) \ne \emptyset$ but $W \cap A \cap f^{-1}(y)$ is meager in $f^{-1}(y)$, then $(W \setminus A) \cap f^{-1}(y) \subseteq (U \setminus A) \cap f^{-1}(y)$ is non-meager (since $f^{-1}(y)$ is Baire), i.e., $y \in \exists^*_f(U \setminus A)$.
\end{proof}

The following result generalizes the well-known fact \cite[22.22]{Kec} that category quantifiers applied to Borel sets in products of Polish spaces preserve Borel complexity.

\begin{theorem}
\label{thm:bairequant-sigma0xi}
Let $f : X -> Y$ be a continuous open map, where $X$ is a second-countable completely Baire space.  Then $X$ is $f$-fiberwise Baire, and for every $A \in \*\Sigma^0_\xi(X)$, we have $\exists^*_f(A) \in \*\Sigma^0_\xi(Y)$.
\end{theorem}
\begin{proof}
Since $X$ is second-countable, so is $f(X) \subseteq Y$, whence points $y \in f(X)$ are $\*\Pi^0_2$, whence fibers $f^{-1}(y) \subseteq X$ for $y \in Y$ are $\*\Pi^0_2$, hence Baire.  Let $\@W$ be a countable basis of open sets in $X$; then $\@W$ is a $f$-fiberwise weak basis for any open $U \subseteq X$.  So the hypotheses of \cref{thm:bairequant-ops} are satisfied.  Now induct on $\xi$, using \cref{thm:bairequant-ops} and the fact that for $\xi > 1$, $\*\Sigma^0_\xi(X)$ consists precisely of sets of the form $\bigcup_n (U_n \setminus B_n)$ with $U_n$ open and $B_n \in \*\Sigma^0_{\zeta_n}(X)$, $\zeta_n < \xi$.
\end{proof}

We also have the following generalization of the classical Kuratowski--Ulam theorem \cite[8.41]{Kec}; the proof is essentially from \cite[A.1]{MT}.  Recall that a continuous map $f : X -> Y$ is \defn{category-preserving} if the preimage of every meager set is meager; this includes all open maps.

\begin{theorem}[Kuratowski--Ulam theorem]
Let $f : X -> Y$ be a continuous open map, where $X$ is a second-countable completely Baire space.  Then for every Baire-measurable $A \subseteq X$,
\begin{enumerate}
\item[(i)]  $A \cap f^{-1}(y)$ is Baire-measurable in $f^{-1}(y)$ for comeagerly many $y \in Y$;
\item[(ii)]  $\exists^*_f(A), \forall^*_f(A) \subseteq Y$ are Baire-measurable;
\item[(iii)]  $\exists^*_f(A) \subseteq Y$ (respectively $\forall^*_f(A) \subseteq Y$) is (co)meager iff $A \subseteq X$ is.
\end{enumerate}
\end{theorem}
\begin{proof}
First, we show $\Longleftarrow$ in (iii).  Let $A \subseteq X$ be comeager.  By \cref{thm:bairequant-ops}(ii), we may assume $A$ is dense open.  Let $\@W$ be a countable basis of open sets in $X$.  Then for each $W \in \@W$, $f(A \cap W)$ is dense open in $f(W)$, since if $\emptyset \ne V \subseteq f(W)$ is open then $\emptyset \ne W \cap f^{-1}(V)$ whence (since $A$ is dense) $\emptyset \ne A \cap W \cap f^{-1}(V)$ whence $\emptyset \ne f(A \cap W \cap f^{-1}(V)) = f(A \cap W) \cap V$.  It follows that
\begin{align*}
G := \bigcap_{W \in \@W} (\neg f(W) \cup f(A \cap W))
\end{align*}
is (a countable intersection of dense $\*\Pi^0_2$ sets, hence) comeager.  We have $y \in G$ iff for every $W \in \@W$ with $W \cap f^{-1}(y) \ne \emptyset$ we have $A \cap W \cap f^{-1}(y) \ne \emptyset$, i.e., iff $A \cap f^{-1}(y)$ is dense in $f^{-1}(y)$.  Thus $G \subseteq \forall^*_f(A)$, and so $\forall^*_f(A)$ is comeager, as desired.

Now let $A \subseteq X$ be Baire-measurable, say $A = U \triangle M$ where $U$ is open and $M$ is meager.  Then for all of the comeagerly many $y \in \forall^*_f(\neg M)$ (by $\Longleftarrow$ in (iii)), we have that $M \cap f^{-1}(y)$ is meager in $f^{-1}(y)$, whence $A \cap f^{-1}(y) = (U \cap f^{-1}(y)) \triangle (M \cap f^{-1}(y))$ is Baire-measurable, proving (i), and $A \cap f^{-1}(y)$ is comeager (or meager) in $f^{-1}(y)$ iff $U \cap f^{-1}(y)$ is.  The latter implies that $\forall^*_f(A) \triangle \forall^*_f(U) \subseteq \exists^*_f(M)$ is meager; by \cref{thm:bairequant-sigma0xi}, $\forall^*_f(U)$ is $\*\Pi^0_2$ and so Baire-measurable, whence $\forall^*_f(A)$ is Baire-measurable, proving (ii).  Similarly, $\exists^*_f(A) \triangle \exists^*_f(U) \subseteq \exists^*_f(M)$ is meager.  Now to prove $\Longrightarrow$ in (iii): if $\exists^*_f(A)$ is meager, then so is $\exists^*_f(U) = f(U)$ (by \cref{thm:bairequant-ops}(i)), whence so is $U \subseteq f^{-1}(f(U))$ since $f$ is category-preserving, whence so is $A$.
\end{proof}

We close this section with some simple applications of Baire category.

\begin{proposition}
\label{thm:openmap-proj-sigma0xi}
Let $f : X -> Y$ be a continuous open map between quasi-Polish spaces.  Then for any $\*\Sigma^0_\xi$ and $f$-fiberwise open $A \subseteq X$, $f(A) \subseteq Y$ is $\*\Sigma^0_\xi$.
\end{proposition}
\begin{proof}
By \cref{thm:bairequant-ops}(i), $f(A) = \exists^*_f(A)$, which is $\*\Sigma^0_\xi$ by \cref{thm:bairequant-sigma0xi}.
\end{proof}

\begin{corollary}
Let $f : X -> Y$ be a continuous open surjection between quasi-Polish spaces.  Then $B \subseteq Y$ is $\*\Sigma^0_\xi$ iff $f^{-1}(B)$ is.
\end{corollary}
\begin{proof}
Since $B$ is surjective, $B = f(f^{-1}(B))$, which is $\*\Sigma^0_\xi$ by \cref{thm:openmap-proj-sigma0xi} if $f^{-1}(B)$ is.
\end{proof}

\begin{theorem}
\label{thm:openmap-sect}
Let $f : X -> Y$ be a continuous open map between quasi-Polish spaces.  Then $f$ admits a Borel \defn{section} $s : f(X) -> X$, i.e., a Borel map such that $f \circ s = 1_{f(X)}$.
\end{theorem}
\begin{proof}
Apply the large section uniformization theorem \cite[18.6]{Kec} to the inverse graph relation of $f$, $R := \{(y, x) \in Y \times X \mid f(x) = y\}$, using the $\sigma$-ideals
\begin{align*}
\@I_y := \{A \subseteq X \mid A \cap f^{-1}(y) \text{ is meager in } f^{-1}(y)\}
\end{align*}
for each $y \in Y$.  Clearly each fiber $R_y = f^{-1}(y)$ is $\emptyset$ or $\not\in \@I_y$; and $y |-> \@I_y$ is Borel-on-Borel (see \cite[18.5]{Kec}), since for every (quasi-)Polish space $Z$ and Borel set $B \subseteq Z \times Y \times X$, we have
\begin{align*}
B_{z,y} := \{x \in X \mid (z, y, x) \in B\} \in \@I_y
&\iff \{x \in f^{-1}(y) \mid (z, y, x) \in B\} \text{ is meager in } f^{-1}(y) \\
&\iff y \not\in \exists^*_f(\{x \in X \mid (z, f(x), x) \in B\}) \\
&\iff (z, y) \not\in \exists^*_{Z \times f}(\{(z, x) \in Z \times X \mid (z, f(x), x) \in B\})
\end{align*}
(where $Z \times f : Z \times X -> Z \times Y$ takes $(z, x)$ to $(z, f(x))$), which is Borel in $(z, y)$ by \cref{thm:bairequant-sigma0xi}.  It follows that $R$ has a Borel uniformizing function $s : f(X) -> X$, which is the desired section.
\end{proof}

A topological space $X$ is \defn{irreducible} if $X \ne \emptyset$, and whenever $X = F \cup G$ with $F, G$ closed, then either $X = F$ or $X = G$.  A topological space $X$ is \defn{sober} if $X$ is $T_0$, and for every irreducible closed $F \subseteq X$, there is a (unique, by $T_0$) $x \in X$ such that $F = \-{\{x\}}$.

\begin{theorem}[{\cite[Corollary~39]{deB}}]
\label{thm:qpol-sober}
Quasi-Polish spaces are sober.
\end{theorem}
\begin{proof}
Let $X$ be quasi-Polish and $F \subseteq X$ be irreducible closed.  Let $\@U$ be a countable basis of open sets in $X$.  For every open $U, V \subseteq X$ which both intersect $F$, by irreducibility, also $F \cap U \cap V \ne \emptyset$.  Thus for every $U \in \@U$ such that $F \cap U \ne \emptyset$, $F \cap U \subseteq F$ is dense.  So by Baire category, there is some $x \in F \cap \bigcap \{U \in \@U \mid F \cap U \ne \emptyset\}$, which is easily seen to satisfy $\-{\{x\}} = F$.
\end{proof}

\section{Posites}
\label{sec:posite}

In this section, we study a special kind of copresentation, one where all of the relations between open sets are of the form ``open sets $V_i$ cover $U$''.

A \defn{posite}\footnote{This notion comes from locale theory; see \cite[II~2.11]{Jstone}.} $(\@U, |>)$ consists of a poset $\@U$ and a binary relation $|>$ between subsets of $\@U$ and elements of $\@U$.  We think of elements $U \in \@U$ as names for basic open sets, and of the relation $\@V = \{V_i\}_i |> U$ for $U, V_i \in \@U$ as meaning ``$\{V_i\}_i$ cover $U$''.  The relation $|>$ is required to satisfy:
\begin{gather}
\label{eq:posite-pcpl}
V \in \@V |> U \implies V \le U, \\
\label{eq:posite-stable}
\@V |> U \ge U' \implies \exists \@V' |> U'\; \forall V' \in \@V'\; \exists V \in \@V\, (V' \le V)
\end{gather}
(the second condition says ``every open cover of $U$ refines to an open cover of $U' \subseteq U$'').

Every posite $(\@U, |>)$ determines a topological space, as follows.  For a poset $\@U$, let
\begin{align*}
\Up(\@U) := \{\@A \in \#S^\@U \mid \forall U \le V \in \@U\, (U \in \@A \implies V \in \@A)\}
\end{align*}
denote the space of upward-closed subsets of $\@U$, and let
\begin{align*}
\Filt(\@U) := \{\@A \in \Up(\@U) \mid \@A \ne \emptyset \AND \forall U, V \in \@A\, \exists W \in \@A\, (W \le U \AND W \le V)\}
\end{align*}
denote the space of filters in $\@U$.  Now for a posite $(\@U, |>)$, let
\begin{align*}
\Coidl(\@U, |>) := \{\@A \in \Up(\@U) \mid \forall \@V |> U\, (U \in \@A \implies \exists V \in \@V\, (V \in \@A))\}
\end{align*}
denote the space of \defn{$|>$-coideals} in $\@U$, i.e., the complements of \defn{$|>$-ideals} $\@A \in \Idl(\@U, |>)$, which are downward-closed subsets $\@A \subseteq \@U$ such that $\forall \@V |> U\, (\@V \subseteq \@A \implies U \in \@A)$.  Finally, let
\begin{align*}
\PFilt(\@U, |>) := \Filt(\@U) \cap \Coidl(\@U, |>) \subseteq \#S^\@U
\end{align*}
denote the space of \defn{$|>$-prime filters} in $\@U$; we call $\PFilt(\@U, |>)$ the \defn{space copresented by $(\@U, |>)$}.  We think of $\@X \in \PFilt(\@U, |>)$ as a ``point'', where $U \in \@X$ are the ``basic neighborhoods'' of $\@X$.

A posite $(\@U, |>)$ is \defn{countable} if both $\@U$ and $|>$ (as a set of pairs) are countable.  In that case, the sets $\Up(\@U), \Filt(\@U), \Coidl(\@U, |>), \PFilt(\@U, |>) \subseteq \#S^\@U$ are $\*\Pi^0_2$, hence quasi-Polish; a $\*\Pi^0_2$ definition, i.e., countable copresentation, of $\PFilt(\@U, |>) \subseteq \#S^\@U$ is given by combining the above definitions of $\Up(\@U), \Filt(\@U), \Coidl(\@U, |>)$.  The key fact about countable posites is the following ``prime ideal theorem'', which says that the copresented spaces have ``enough points'':

\begin{theorem}[see {\cite[3.14]{Hec}}]
\label{thm:posite-ctb-pit}
Let $(\@U, |>)$ be a countable posite.  Then for every $\@A \in \Coidl(\@U, |>)$ and $W \in \@A$, there is a $\@X \in \PFilt(\@U, |>)$ such that $W \in \@X \subseteq \@A$.
\end{theorem}
\begin{proof}
Let
\begin{align*}
K := \{\@B \in \Filt(\@U) \mid W \in \@B \subseteq \@A\};
\end{align*}
then $K = \Filt(\@U) \cap \up \{W\} \cap \-{\{\@A\}} \subseteq \#S^\@U$ is $\*\Pi^0_2$.  For each $\@V |> U$, let
\begin{align*}
Q_{\@V,U} := \{\@B \in \Up(\@U) \mid U \in \@B \implies \exists V \in \@V\, (V \in \@B)\};
\end{align*}
then $Q_{\@V,U} \subseteq \#S^\@U$ is $\*\Pi^0_2$.  We claim that each $Q_{\@V,U}$ is dense in $K$.  This will imply by Baire category that $\emptyset \ne K \cap \bigcap_{\@V |> U} Q_{\@V,U} = K \cap \Coidl(\@U, |>) = \{\@X \in \PFilt(\@U, |>) \mid W \in \@X \subseteq \@A\}$, as desired.

To prove the claim, let $\up \@C \subseteq \#S^\@U$ be a basic open set for some finite $\@C \subseteq \@U$ such that $K \cap \up \@C \ne \emptyset$; we must show that $K \cap \up \@C \cap Q_{\@V,U} \ne \emptyset$.  Let $\@D \in K \cap \up \@C$, i.e., $\@D \in \Filt(\@U)$ with $W \in \@D \subseteq \@A$ and $\@C \subseteq \@D$.  If $U \not\in \@D$, then clearly $\@D \in Q_{\@V,U}$, so we are done.  Otherwise, $U \in \@D$, so since $\@D$ is a filter, there is some $U' \in \@D$ with $U' \le U, W$ and $U' \le C$ for all $C \in \@C$.  By \eqref{eq:posite-stable}, there is some $\@V' |> U'$ such that for every $V' \in \@V'$ there is some $V \in \@V$ with $V' \le V$.  Since $U' \in \@D \subseteq \@A$ and $\@A \in \Coidl(\@U, |>)$, there is some $V' \in \@V$ with $V' \in \@A$.  Then it is easily verified that $\up V' \in K \cap \up \@C \cap Q_{\@V,U}$.
\end{proof}

\begin{corollary}
\label{thm:posite-ctb-spatial}
Let $(\@U, |>)$ be a countable posite.  Then we have a bijection
\begin{align*}
\Idl(\@U, |>) &\cong \{\text{open subsets of $\PFilt(\@U, |>)$}\} \\
\@A &|-> \bigcup_{U \in \@A} (\PFilt(\@U, |>) \cap \up \{U\}) \\
\{U \in \@U \mid \PFilt(\@U, |>) \cap \up \{U\} \subseteq C\} &<-| C.
\end{align*}
\end{corollary}
\begin{proof}
An open set $C \subseteq \PFilt(\@U, |>)$ is a union of basic open sets $\PFilt(\@U, |>) \cap \up \@S$ for finite $\@S \subseteq \@U$; since $\PFilt(\@U, |>)$ consists of filters, we have
\begin{align*}
\PFilt(\@U, |>) \cap \up \@S = \bigcup_{U \in \bigcap_{V \in \@S} \down V} (\PFilt(\@U, |>) \cap \up \{U\}),
\end{align*}
which easily implies that the two maps compose to the identity on the right-hand side.

For the other composite, let $\@A \in \Idl(\@U, |>)$; we must show that
\begin{align*}
\@A = \{V \in \@U \mid \PFilt(\@U, |>) \cap \up \{V\} \subseteq \bigcup_{U \in \@A} \up \{U\}\}.
\end{align*}
$\subseteq$ is obvious.  Conversely, for $V \not\in \@A$, by \cref{thm:posite-ctb-pit} there is some $\@X \in \PFilt(\@U, |>)$ such that $V \in \@X \subseteq \neg \@A$; then $\@X \in \PFilt(\@U, |>) \cap \up \{V\} \setminus \bigcup_{U \in \@A} \up \{U\}$.
\end{proof}

Now let $X$ be a $T_0$ space.  A \defn{basic posite for $X$} is a posite $(\@U, |>)$ where $\@U$ is a basis of open sets in $X$ (ordered by inclusion) and such that the canonical embedding $e_\@U : X -> \#S^\@U$ has image $\PFilt(\@U, |>)$, thus exhibiting $X$ as (a homeomorphic copy of) the space copresented by $(\@U, |>)$:
\begin{align}
\label{eq:posite-copres}
e_\@U : X &\cong \PFilt(\@U, |>) \subseteq \#S^\@U.
\end{align}
Note that since $\@U$ is a basis, we always have $e_\@U(X) \subseteq \Filt(\@U, |>)$.  The condition $e_\@U(X) \subseteq \Coidl(\@U, |>)$ (equivalently, $e_\@U(X) \subseteq \PFilt(\@U, |>)$) is equivalent to
\begin{align}
\label{eq:posite-subcan}
\@V |> U \implies \bigcup \@V = U, \tag{\ref*{eq:posite-copres}a}
\end{align}
i.e., that the covering relations specified by $|>$ actually hold in $X$.

\begin{proposition}
\label{thm:qpol-posite-ctb}
Let $X$ be a quasi-Polish space.  For any countable open basis $\@U$ for $X$, there is a countable basic posite $(\@U, |>)$ for $X$.
\end{proposition}
\begin{proof}
By \cref{thm:qpol-sub=pi02}, $e_\@U(X) \subseteq \#S^\@U$ is $\*\Pi^0_2$.  Let $e_\@U(X) = \bigcap_i (\neg S_i \cup T_i)$ where $S_i, T_i \subseteq \#S^\@U$ are open.  We may assume that the $S_i$ are basic open, i.e.,
\begin{align*}
S_i = \up \@U_i
\end{align*}
for some finite $\@U_i \subseteq \@U$.  Write each $T_i$ as a countable union of basic open sets
\begin{align*}
T_i = \bigcup_{\@V \in \&V_i} \up \@V
\end{align*}
for some countable set $\&V_i$ of finite $\@V \subseteq \@U$.  Let $|>$ consist of the relations
\begin{align*}
\@V_{i,U} := \{V \in \@U \mid \exists \@V \in \&V_i\, (V \subseteq U \cap \bigcap \@V)\} |> U
\end{align*}
for each $\@U \ni U \subseteq \bigcap \@U_i$.

To check \eqref{eq:posite-stable}: for $\@V_{i,U} |> U \supseteq U'$, it is easily seen that $\@V_{i,U'} |> U'$.

To check \eqref{eq:posite-subcan} (which implies \eqref{eq:posite-pcpl}): for $\@V_{i,U} |> U$, we have $\bigcup \@V_{i,U} = \bigcup_{\@V \in \&V_i} (U \cap \bigcap \@V) = U \cap \bigcup_{\@V \in \&V_i} \bigcap \@V = U \cap e_\@U^{-1}(T_i)$; since $U \subseteq \bigcap \@U_i = e_\@U^{-1}(S_i)$ and $U \subseteq X \subseteq e_\@U^{-1}(\neg S_i \cup T_i)$, we have $U \subseteq e_\@U^{-1}(T_i)$, whence $\bigcup \@V_{i,U} = U$.

Finally, to check $\supseteq$ in \eqref{eq:posite-copres}: let $\@X \in \PFilt(\@U, |>)$; we must show $\@X \in \bigcap_i (\neg S_i \cup T_i) = e_\@U(X)$.  If $\@X \in S_i = \up \@U_i$, i.e., $\@U_i \subseteq \@X$, then since $\@X$ is a filter, there is some $U \in \@X$ with $U \subseteq \bigcap \@U_i$; since $\@X$ is a $|>$-coideal and $\@V_{i,U} |> U$, there is some $\@V \in \&V_i$ and $V \subseteq U \cap \bigcap \@V$ with $V \in \@X$, whence $\@V \subseteq \@X$ since $\@X$ is upward-closed, whence $\@X \in \up \@V \subseteq T_i$.
\end{proof}

\begin{corollary}
A topological space $X$ is quasi-Polish iff it is homeomorphic to the space copresented by a countable posite.  \qed
\end{corollary}

\section{Lower powerspaces}
\label{sec:lowpow}

Let $X$ be a topological space.  The \defn{lower powerspace} $\@F(X)$ is the space of closed sets in $X$, with topology generated by the subbasic open sets
\begin{align*}
\Dia U := \{F \in \@F(X) \mid F \cap U \ne \emptyset\}
\end{align*}
for open sets $U \subseteq X$.

We have a canonical map
\begin{align*}
\down : X &--> \@F(X) \\
x &|--> \-{\{x\}},
\end{align*}
such that $\down^{-1}(\Dia U) = U$; thus $\down$ is continuous, and an embedding if $X$ is $T_0$.

If $X$ is quasi-Polish, then it follows from \cref{thm:qpol-sub-pi02} that $\down(X) \subseteq \@F(X)$ is $\*\Pi^0_2$.  A simple $\*\Pi^0_2$ definition is provided by

\begin{proposition}[{\cite[Proposition~3]{dBK}}]
\label{thm:down-pi02}
If $X$ is second-countable and sober, then $\down(X) \subseteq \@F(X)$ is $\*\Pi^0_2$.
\end{proposition}
\begin{proof}
Let $\@U$ be a countable basis of open sets in $X$.  Then
\begin{align*}
\down(X) = \{F \in \@F(X) \mid F \in \Dia X \AND \forall U, V \in \@U\, (F \in \Dia U \cap \Dia V \implies F \in \Dia(U \cap V))\},
\end{align*}
where the right-hand is easily seen to consist precisely of the irreducible closed sets.
\end{proof}

It is easily seen that the specialization order on $\@F(X)$ is inclusion.  Thus, for $X \ne \emptyset$, $\@F(X)$ is never $T_1$; in particular, $\@F(X)$ is never Polish, even if $X$ is.  One of the main advantages of working with quasi-Polish spaces is

\begin{theorem}[{\cite[Theorem~5]{dBK}}]
\label{thm:qpol-lowpow}
If $X$ is quasi-Polish, then so is $\@F(X)$.
\end{theorem}
\begin{proof}
Let $(\@U, |>)$ be a countable basic posite for $X$ (\cref{thm:qpol-posite-ctb}).  By \cref{thm:posite-ctb-spatial}, we have a bijection
\begin{align*}
f_\@U : \@F(X) &\cong \Coidl(\@U, |>) \subseteq \#S^\@U \\
F &|-> \{U \in \@U \mid F \cap U \ne \emptyset\},
\end{align*}
with $f_\@U^{-1}(\Coidl(\@U, |>) \cap \up \{U\}) = \Dia U$, whence $f_\@U$ is a homeomorphism.
\end{proof}

Note that the subbasic open sets $\Dia U \subseteq \@F(X)$ are the usual generators of the Effros Borel structure \cite[\S12.C]{Kec}; thus, the Effros Borel space is the underlying Borel space of $\@F(X)$.  So we have

\begin{corollary}
If $X$ is quasi-Polish, then the Effros Borel space of $X$ is standard Borel.  \qed
\end{corollary}

Combined with \cref{thm:lqpol-qpol} and \cref{thm:pol-qpol}, this implies Tserunyan's result \cite[Theorem~A]{Lup} that the Effros Borel space of a $\sigma$-locally Polish space is standard Borel.

\section{Continuous open surjections}
\label{sec:opensurj}

Whereas a continuous open metrizable image of a Polish space is Polish \cite[8.19]{Kec}, in the quasi-Polish context we have

\begin{theorem}[{\cite[Theorem~40]{deB}}]
\label{thm:qpol-opensurj}
Let $X$ be a quasi-Polish space and $f : X -> Y$ be a continuous open surjection onto a $T_0$ space $Y$.  Then $Y$ is quasi-Polish.
\end{theorem}

We will prove this using the lower powerspace $\@F(X)$.  Before doing so, we make some general remarks on the connection between lower powerspaces, open maps, and the more general class of \emph{essential maps}, defined below.

Let $X$ be a topological space.  A subset $A \subseteq X$ is \defn{saturated} if it is upward-closed in the specialization preorder; the \defn{saturation} ${\Uparrow}A$ of an arbitrary subset $A \subseteq X$ is its upward closure.  Every open set is saturated; and the saturation ${\Uparrow} A$ of $A$ is the intersection of all open sets containing $A$.  A continuous map $f : X -> Y$ is \defn{essential} if for every open $U \subseteq X$, ${\Uparrow} f(U) \subseteq Y$ is open.  In particular, every continuous open map is essential.

\begin{lemma}
\label{thm:essmap-lowpow}
A continuous map $f : X -> Y$ is essential iff
\begin{align*}
f^{-1} \circ \down : Y &--> \@F(X) \\
y &|--> f^{-1}(\-{\{y\}})
\end{align*}
is continuous.
\end{lemma}

(Recall (\cref{sec:lowpow}) that $\down : Y -> \@F(Y)$ denotes the map $\down y := \-{\{y\}}$.)

\begin{proof}
For open $U \subseteq X$, we have
\begin{align*}
f^{-1}(\-{\{y\}}) \in \Dia U
&\iff U \not\subseteq f^{-1}(\neg \-{\{y\}}) \\
&\iff f(U) \not\subseteq \neg \-{\{y\}} \\
&\iff {\Uparrow} f(U) \not\subseteq \neg \-{\{y\}} \qquad\text{since $\neg \-{\{y\}}$ is saturated} \\
&\iff y \in {\Uparrow} f(U),
\end{align*}
i.e., $(f^{-1} \circ \down)^{-1}(\Dia U) = {\Uparrow} f(U)$.
\end{proof}

\begin{lemma}
\label{thm:openmap-lowpow}
Let $f : X -> Y$ be a continuous open map.  Then
\begin{align*}
f^{-1} \circ \down = \-{f^{-1}} : Y &--> \@F(X) \\
y &|--> \-{f^{-1}(y)},
\end{align*}
hence $\-{f^{-1}}$ is continuous.
\end{lemma}
\begin{proof}
Let $y \in Y$.  Clearly $f^{-1}(y) \subseteq f^{-1}(\-{\{y\}})$, and the latter is closed since $f$ is continuous, whence $\-{f^{-1}(y)} \subseteq f^{-1}(\-{\{y\}})$.  Conversely, we have
\begin{align*}
f^{-1}(\-{\{y\}}) \subseteq \-{f^{-1}(y)}
&\iff \neg \-{f^{-1}(y)} \subseteq f^{-1}(\neg \-{\{y\}}) \\
&\iff f(\neg \-{f^{-1}(y)}) \subseteq \neg \-{\{y\}} \\
&\iff f(\neg \-{f^{-1}(y)}) \subseteq \neg \{y\} \qquad\text{since $f(\neg \-{f^{-1}(y)})$ is open} \\
&\iff \neg \-{f^{-1}(y)} \subseteq f^{-1}(\neg \{y\}) \\
&\iff f^{-1}(y) \subseteq \-{f^{-1}(y)}
\end{align*}
which is clearly true.
\end{proof}

\begin{proof}[Proof of \cref{thm:qpol-opensurj}]
Consider the map $\-{f^{-1}} : Y -> \@F(X)$.  By the above lemmas, for open $U \subseteq X$ we have
\begin{align*}
\-{f^{-1}}^{-1}(\Dia U) = f(U);  \tag{$*$}
\end{align*}
thus since $f$ is open surjective and $Y$ is $T_0$, $\-{f^{-1}}$ is an embedding.  Let $\@U$ be a countable basis of open sets in $X$.  We claim
\begin{align*}
\-{f^{-1}}(Y) = \left\{F \in \@F(X) \relmiddle| \begin{aligned}
&F \in \Dia X \AND \\
&\forall U, V \in \@U\, (F \in \Dia U \cap \Dia V \implies F \in \Dia(f^{-1}(f(U)) \cap V)) \AND \\
&\forall U \in \@U\, (F \in \Dia f^{-1}(f(U)) \implies F \in \Dia U)
\end{aligned}\right\}.
\end{align*}
$\subseteq$ is straightforward.  To prove $\supseteq$, let $F$ belong to the right-hand side.  By the first condition on $F$, $F \ne \emptyset$.  By the second condition on $F$, for every $U \in \@U$ with $F \in \Dia U$, $f^{-1}(f(U))$ is dense in $F$.  Thus by Baire category, there is some
\begin{align*}
x \in F \cap \bigcap \{f^{-1}(f(U)) \mid U \in \@U \AND F \in \Dia U\}.
\end{align*}
We claim that $F = \-{f^{-1}(f(x))}$.  Indeed, for $U \in \@U$, we have
\begin{align*}
\-{f^{-1}(f(x))} \in \Dia U
&\iff f(x) \in f(U) \qquad\text{by ($*$)} \\
&\iff x \in f^{-1}(f(U)).
\end{align*}
If $x \in f^{-1}(f(U))$, then $x \in F \cap f^{-1}(f(U)) \ne \emptyset$, i.e., $F \in \Dia f^{-1}(f(U))$, whence $F \in \Dia U$ by the third condition on $F$.  Conversely, if $F \in \Dia U$, then $x \in f^{-1}(f(U))$ by definition of $x$.  Thus $F \in \Dia U \iff \-{f^{-1}(f(x))} \in \Dia U$ for every $U \in \@U$.
\end{proof}

We also have the following ``converse'' of \cref{thm:qpol-opensurj}, which generalizes the fact \cite[7.14]{Kec} that every nonempty Polish space is a continuous open image of Baire space $\#N^\#N$:

\begin{theorem}[{\cite[Lemma~38]{deB}}]
\label{thm:qpol-opensurj-baire}
Every nonempty quasi-Polish space is a continuous open image of $\#N^\#N$.
\end{theorem}
\begin{proof}
It is easy to see that $f : \#N^\#N -> \#S$ given by $f(x) = 0 \iff x = (0, 0, \dotsc)$ is a continuous open surjection.  It follows that $g := f^\#N : \#N^\#N \cong (\#N^\#N)^\#N -> \#S^\#N$ is a continuous open surjection.  Now for a $\*\Pi^0_2$ subset $X \subseteq \#S^\#N$, $h := g|g^{-1}(X) : g^{-1}(X) -> X$ is a continuous open surjection (with $h(g^{-1}(X) \cap U) = X \cap h(U)$ for open $U \subseteq \#N^\#N$).  If $X \ne \emptyset$, then $g^{-1}(X) \ne \emptyset$, whence by \cite[7.14]{Kec}, there is a continuous open surjection $k : \#N^\#N -> g^{-1}(X)$, whence $h \circ k : \#N^\#N -> X$ is a continuous open surjection.
\end{proof}

\begin{corollary}[{\cite[Corollary~42]{deB}}]
A nonempty space $X$ is quasi-Polish iff it is a continuous open $T_0$ image of $\#N^\#N$.  \qed
\end{corollary}

\section{The convergent strong Choquet game}
\label{sec:game}

We conclude with a game characterization of quasi-Polish spaces, analogous to that of Polish spaces via the strong Choquet game \cite[8.18]{Kec}.  This characterization is from \cite[Section~10]{deB}.

Let $X$ be a topological space.  The \defn{convergent strong Choquet game}\footnote{The name is derived from \cite{DM}, where the game is studied for $T_1$ spaces $X$.} $\@G(X)$ on $X$ is played in exactly the same way as the strong Choquet game \cite[8.14]{Kec}, but with different winning conditions.  That is, players I and II alternate turns, with I moving first:
\begin{align*}
\renewcommand\arraystretch{1.5}
\begin{array}{c|ccccc}
\text{I} & (U_0,x_0) && (U_1,x_1) && \dotsb \\
\text{II} & & V_0 && V_1
\end{array}
\end{align*}
On turn $n$, I must play an open set $U_n \subseteq X$ and a point $x_n \in U_n$, with $U_n \subseteq V_{n-1}$ if $n \ge 1$; and II must respond with an open set $V_n \subseteq X$ such that $x_n \in V_n \subseteq U_n$.  Player II wins iff $X = \emptyset$ (so I is unable to play the first move) or the open sets $U_0 \supseteq V_0 \supseteq U_1 \supseteq V_1 \supseteq \dotsb$ form a neighborhood basis of some $x \in X$.  Note that the latter condition is equivalent to:
\begin{enumerate}
\item[(i)] $x \in \bigcap_n U_n = \bigcap_n V_n$, and
\item[(ii)] the filter base $\{U_n\}_n$ (equivalently, $\{V_n\}_n$) converges to $x$.
\end{enumerate}
If II has a winning strategy, then we call $X$ a \defn{convergent strong Choquet space}.

In the strong Choquet game, the winning condition for player II is that $\bigcap_n U_n = \bigcap_n V_n \ne \emptyset$, which is weaker than that in the convergent strong Choquet game; thus

\begin{proposition}
\label{thm:cscs-scs}
Every convergent strong Choquet space is a strong Choquet space.  \qed
\end{proposition}

The usual proof that completely metrizable spaces are strong Choquet in fact shows that

\begin{proposition}
\label{thm:cmet-cscs}
Completely metrizable spaces are convergent strong Choquet.  \qed
\end{proposition}

\begin{remark}
It is easily seen that in a convergent strong Choquet space, the set of points with a countable neighborhood basis must be dense.  Thus, for example, $2^I$ for uncountable $I$ is \emph{not} convergent strong Choquet.  In particular, unlike with strong Choquet spaces, \emph{not every compact Hausdorff space} is convergent strong Choquet; and convergent strong Choquet spaces are \emph{not closed under uncountable products}.
\end{remark}

Convergent strong Choquet spaces share many of the closure properties of quasi-Polish spaces:

\begin{proposition}
\label{thm:cscs-pi02}
A $\*\Pi^0_2$ subspace of a convergent strong Choquet space is convergent strong Choquet.
\end{proposition}
\begin{proof}
Let $X$ be convergent strong Choquet and $Y = \bigcap_{n \in \#N} (\neg A_n \cup B_n) \in \*\Pi^0_2(X)$, where $A_n, B_n \subseteq X$ are open.  Let II play according to the following strategy in $\@G(Y)$.
\begin{align*}
\@G(Y) &
\renewcommand\arraystretch{1.5}
\begin{array}{c|ccccc}
\text{I} & (U_0,x_0) && (U_1,x_1) && \dotsb \\
\text{II} & & V_0 && V_1
\end{array}
\\[2em]
\@G(X) &
\renewcommand\arraystretch{1.5}
\begin{array}{c|ccccc}
\text{I} & (U'_0,x_0) && (U'_1,x_1) && \dotsb \\
\text{II} & & V'_0 && V'_1
\end{array}
\end{align*}
II keeps a side copy of the game $\@G(X)$ running.  Each move $V_k$ of II in $\@G(Y)$ will be determined by the corresponding move $V'_k$ of II in $\@G(X)$ via
\begin{align*}
V_k := V'_k \cap Y.  \tag{$*$}
\end{align*}
On turn $k$, I plays $(U_k, x_k)$ in $\@G(Y)$.  Let $U'_k \subseteq X$ be open so that
\begin{enumerate}
\item[(i)]  $x_k \in U'_k \cap Y \subseteq U_k$;
\item[(ii)] $U'_k \subseteq V'_{k-1}$ if $k \ge 1$ (possible since $x_k \in U_k \subseteq V_{k-1} \subseteq V'_{k-1}$ by ($*$));
\item[(iii)]  for all $n \le k$ such that $x_k \in A_n$, we have $U'_k \subseteq B_n$ (possible since $x_k \in Y \subseteq \neg A_n \cup B_n$).
\end{enumerate}
Let I play $(U'_k, x_k)$ in $\@G(X)$; this is legal by (i) and (ii).  Let $V'_k$ be given by II's winning strategy in $\@G(X)$, so that $x_k \in V'_k \subseteq U'_k$.  II then plays $V_k := V'_k \cap Y$ (as per ($*$)) in $\@G(Y)$; this is legal since clearly $x_k \in V_k$, and $V_k = V'_k \cap Y \subseteq U'_k \cap Y \subseteq U_k$ by (i).

To check that II wins $\@G(Y)$: since II wins $\@G(X)$, there is some $x \in X$ such that $\{V'_k\}_k$ forms a neighborhood basis for $x$; by ($*$), it is enough to check that $x \in Y = \bigcap_n (\neg A_n \cup B_n)$.  Fix $n \in \#N$; we check that $x \in \neg A_n \cup B_n$.  If $x_k \in A_n$ for some $k \ge n$, then by (iii), we have $U'_k \subseteq B_n$, whence $x \in V'_k \subseteq U'_k \subseteq B_n$.  Otherwise, for all $k \ge n$ we have $x_k \in \neg A_n$; we have $\lim_{k -> \infty} x_k = x$ (since for every basic neighborhood $V'_k \ni x$ we have $x_l \in V'_l \subseteq V'_k$ for all $l \ge k$), so since $\neg A_n$ is closed, $x \in \neg A_n$.
\end{proof}

\begin{proposition}
\label{thm:cscs-prod}
A countable product of convergent strong Choquet spaces is convergent strong Choquet.
\end{proposition}
\begin{proof}
The proof is similar to the usual proof that products of strong Choquet spaces are strong Choquet (see e.g., \cite[4.1.2(c)]{Gao}).
Let $X_n$ for $n \in \#N$ be convergent strong Choquet spaces and put $X := \prod_n X_n$.  Player II plays in $\@G(X)$ as follows, while keeping track of integers $0 < m_0 < m_1 < m_2 < \dotsb$ and running side games of $\@G(X_n)$ for each $n$, such that on move $k$ in $\@G(X)$, side games $\@G(X_0), \dotsc, \@G(X_{m_k-1})$ are being played.  On turn $k$, after I plays $(U^k, x^k)$ in $\@G(X)$, II finds $m_k > m_{k-1}$ (where $m_{-1} := 0$) such that $U^k$ contains a basic open neighborhood of $x^k$ which is trivial in all but the first $m_k$ coordinates, i.e.,
\begin{align*}
x^k \in U^k_0 \times \dotsb \times U^k_{m_k-1} \times \prod_{n \ge m_k} X_n \subseteq U^k
\end{align*}
for open sets $U^k_0 \subseteq X_0$, $\dotsc$, $U^k_{m_k-1} \subseteq X_{m_k-1}$.  Let $x^k = (x^k_0, x^k_1, \dotsc)$.  I then plays $(U^k_n, x^k_n)$ in $\@G(X_n)$ for each $n < m_k$.  Let $V^k_n$ be given by II's winning strategy in $\@G(X_n)$.  II then plays
\begin{align*}
V^k := V^k_0 \times \dotsb \times V^k_{m_k-1} \times \prod_{n \ge m_k} X_n
\end{align*}
in $\@G(X)$.  It is straightforward to check that this works.
\end{proof}

\begin{proposition}
\label{thm:cscs-opensurj}
A continuous open image of a convergent strong Choquet space is convergent strong Choquet.
\end{proposition}
\begin{proof}
Again, the proof is similar to the usual proof for strong Choquet spaces (see \cite[4.1.2(b)]{Gao}).
Let $X$ be convergent strong Choquet and $f : X -> Y$ be a continuous open surjection.  II plays in $\@G(Y)$ as follows, while running a side game $\@G(X)$.  On turn $k$, I plays $(U_k, y_k)$ in $\@G(Y)$.  Let I play $(U'_k, x_k)$ in $\@G(X)$, for any $x_k \in U'_k \subseteq f^{-1}(U_k)$ with $f(x_k) = y_k$ and $U'_k \subseteq V'_{k-1}$ if $k \ge 1$; the latter is possible since $y_k \in U_k \subseteq V_{k-1} = f(V'_{k-1})$ (see definition of $V_k$ below).  Let $V'_k$ be given by II's winning strategy in $\@G(X)$, so $x_k \in V'_k \subseteq U'_k$.  Then II plays $V_k := f(V'_k)$ in $\@G(Y)$, which is allowed since $y_k = f(x_k) \in f(V'_k) \subseteq f(U'_k) \subseteq U_k$.  Since II wins $\@G(X)$, there is some $x \in X$ with neighborhood basis $\{V'_k\}_k$.  Then $\{V_k\}_k$ is a neighborhood basis of $y := f(x)$: clearly $y \in V_k = f(V'_k)$ for every $k$; and for any open neighborhood $W \ni y$, we have $x \in f^{-1}(W)$, whence there is some $V'_k \subseteq f^{-1}(W)$, whence $V_k = f(V'_k) \subseteq W$.
\end{proof}

We now have the following characterization of quasi-Polish spaces:

\begin{theorem}[{\cite[Theorem~51]{deB}}]
\label{thm:qpol-cscs}
A topological space $X$ is quasi-Polish iff it is $T_0$, second-countable, and convergent strong Choquet.
\end{theorem}
\begin{proof}
$\Longrightarrow$ follows from \cref{thm:cscs-pi02,thm:cscs-prod} and the obvious fact that $\#S$ is convergent strong Choquet (or alternatively \cref{thm:cmet-cscs,thm:qpol-opensurj-baire,thm:cscs-opensurj}).

$\Longleftarrow$:
Let $\tau$ be a winning strategy for II, which we regard as the (uncountably branching) tree of all finite initial runs of the game $(U_0, x_0), V_0, (U_1, x_1), V_1, \dotsc$ where II follows the strategy; for each such run, we define by convention $V_{-1} := X$.
Let $\@U$ be a countable basis of open sets in $X$.
We may find a countably branching subtree $T \subseteq \tau$ such that
\begin{enumerate}
\item[(i)]  I only plays basic open sets in $\@U$; and
\item[(ii)]  for each run of even length in $T$, ending in $V_n$, and for each $V_n \supseteq U_{n+1} \in \@U$, $U_{n+1}$ is the union of all $V_{n+1}$ such that for some $x_{n+1}$, the extension of the run by $(U_{n+1}, x_{n+1}), V_{n+1}$ is in $T$.
\end{enumerate}
To define $T$, start by including the empty run.
Inductively after a run ending in $V_n$, for each $V_n \supseteq U_{n+1} \in \@U$, since I may play $(U_{n+1}, x_{n+1})$ for any $x_{n+1} \in V_n$, the set of all of II's responses $V_{n+1}$ to all such moves, according to $\tau$, form an open cover of $U_{n+1}$; include in $T$ countably many such extensions of the run by $(U_{n+1}, x_{n+1}), V_{n+1}$ corresponding to a countable subcover of $U_{n+1}$.

Now let as usual $[T]$ be the Polish space of infinite branches through $T$.
For each $t = ((U_0, x_0), V_0, (U_1, x_1), V_1, \dotsc) \in [T]$, since II wins, $U_0 \supseteq V_0 \supseteq U_1 \supseteq V_1 \supseteq \dotsb$ forms a neighborhood basis for some $f(t) \in X$, which is unique because $X$ is $T_0$.
So we have defined a function
\begin{equation*}
f : [T] --> X.
\end{equation*}
Clearly $f$ is continuous: for any $t = ((U_0, x_0), V_0, (U_1, x_1), V_1, \dotsc) \in [T]$ and basic neighborhood $V_n$ of $f(t)$, the basic neighborhood
\begin{align*}
N_{(U_0, x_0), V_0, \dotsc, V_n} := \{t' \in [T] \mid t' \text{ extends } (U_0, x_0), V_0, \dotsc, V_n\}
\end{align*}
of $t$ maps into $V_n$.
We now claim that
\begin{itemize}
\item[($*$)]  for any $((U_0, x_0), V_0, \dotsc, V_n) \in T$ (including $n = -1$), we have $f(N_{(U_0, x_0), V_0, \dotsc, V_n}) = V_n$.
\end{itemize}
$\subseteq$ is clear from the definition of $f$.
Conversely, for any $x \in V_n$, we may inductively find a branch $t = ((U_0, x_0), V_0, \dotsc, V_n, (U_{n+1}, x_{n+1}), V_{n+1}, \dotsc) \in N_{(U_0, x_0), V_0, \dotsc, V_n}$, where the $U_{n+1} \supseteq U_{n+2} \supseteq \dotsb$ are a neighborhood basis of $x$ (whence $f(t) = x$), and the $x_{n+1}, x_{n+2}, \dotsc$ are chosen using (ii) above so that $U_{n+1} \supseteq V_{n+1} \ni x$, $U_{n+2} \supseteq V_{n+2} \ni x$, \ldots.
This shows that $f$ is open, as well as surjective (by taking $n = -1$ in ($*$)), whence $X$ is quasi-Polish by \cref{thm:qpol-opensurj}.
\end{proof}

\bigskip\noindent
Department of Mathematics \\
University of Illinois at Urbana--Champaign \\
Urbana, IL 61801 \\
\medskip
\nolinkurl{ruiyuan@illinois.edu}


\begin{thebibliography}{00000}

\bibitem[deB]{deB}  M.~de~Brecht, \emph{Quasi-Polish spaces}, Ann.\ Pure Appl.\ Logic \textbf{164(3)}~(2013), 356--381.

\bibitem[dB2]{dB2}  M.~de~Brecht, \emph{A generalization of a theorem of Hurewicz for quasi-Polish spaces}, Log.\ Methods Comput.\ Sci.\ \textbf{14(1:13)}~(2018), 1--18.

\bibitem[dBK]{dBK}  M.~de~Brecht and T.~Kawai, \emph{On the commutativity of the powerspace constructions}, Log.\ Methods Comput.\ Sci.\ \textbf{15(3)}~(2019), Paper~No.~13, 25pp.

\bibitem[DM]{DM}  F.~Dorais and C.~Mummert, \emph{Stationary and convergent strategies in Choquet games}, Fund.\ Math.\ \textbf{209}~(2010), 59--79.

\bibitem[Gao]{Gao}  S.~Gao, \emph{Invariant descriptive set theory}, Pure and applied mathematics, vol.~293, CRC Press, 2009.

\bibitem[Hec]{Hec}  R.~Heckmann, \emph{Spatiality of countably presentable locales (proved with the Baire category theorem)}, Math.\ Structures Comput.\ Sci.\ \textbf{25(7)}~(2015), 1607--1625.

\bibitem[Joh]{Jstone}  P.~T.~Johnstone, \emph{Stone spaces}, Cambridge Studies in Advanced Mathematics, vol.~3, Cambridge University Press, 1982.

\bibitem[Kec]{Kec}  A.~S.~Kechris, \emph{Classical descriptive set theory}, Graduate Texts in Mathematics, vol.~156, Springer-Verlag, 1995.

\bibitem[Lup]{Lup}  M.~Lupini, \emph{Polish groupoids and functorial complexity}, Trans.\ Amer.\ Math.\ Soc.\ \textbf{369(9)}~(2017), 6683--6723.

\bibitem[MT]{MT}  J.~Melleray and T.~Tsankov, \emph{Generic representations of abelian groups and extreme amenability}, Israel J.\ Math.\ \textbf{198(1)}~(2013), 129--167.

\bibitem[Sel]{Sel}  V.~Selivanov, \emph{Towards a descriptive set theory for domain-like structures}, Theoret.\ Comput.\ Sci.\ \textbf{365(3)}~(2006), 258--282.

\end{thebibliography}
\end{document}